\documentclass{article}
\usepackage[english]{babel}
\usepackage{amsmath,amssymb,latexsym}

\newcommand{\infixand}{\text{ and }}
\newcommand{\infixor}{\text{ or }}
\newcommand{\mathd}{\mathrm{d}}
\newcommand{\nocomma}{}
\newcommand{\nosymbol}{}

\newcommand{\tmdummy}{$\mbox{}$}
\newcommand{\tmname}[1]{\textsc{#1}}
\newcommand{\tmop}[1]{\ensuremath{\operatorname{#1}}}

\newenvironment{proof}{\noindent\textbf{Proof\ }}{\hspace*{\fill}$\Box$\medskip}
\newtheorem{lemma}{Lemma}
\newtheorem{theorem}{Theorem}

\begin{document}

\title{The dyadic and the continuous Hilbert transforms with values in Banach
spaces}

\author{Komla Domelevo and Stefanie Petermichl\footnote{Partially supported by ERC project
CHRiSHarMa no. DLV-682402 and the Alexander von Humboldt foundation}}

\maketitle

\begin{abstract}
  We show that if the Hilbert transform with values in a Banach space is $L^p$
  bounded, then so is the so called dyadic Hilbert transform, a new dyadic shift operator, with a linear relation of
  the bounds. 
\end{abstract}

\section{Introduction}

It is a known fact that the Hilbert transform with values in a Banach space is
bounded if and only if the Banach space has the UMD property. The relationships of the norms, however, remains
unclear. If the $L^p$ norm of the Hilbert transform is noted $h_p$ and the UMD
constant is noted $m_p$, then it is known that
\[ m^{1/2}_p \lesssim h_p \lesssim m^2_p . \]
The lower estimate on the left hand side is due to {\tmname{Bourgain}}
{\cite{Bou1983a}} and the estimate on the right hand side is due to
{\tmname{Burkholder}} {\cite{Bur1983a}}. It is an open question whether these
relations can be improved, ideally if they are linear.

\

Using the classical Haar shift of {\tmname{Petermichl}} \cite{Pet2000}, it was shown by
{\tmname{Petermichl}}--{\tmname{Pott}} {\cite{PetPot2003a}} that if
$s^{\tmop{cl}}_p$ is the $L^p$ bound for the classical Haar shift, then $s^{\tmop{cl}}_p
\leqslant m^2_p$. Their short argument gives an alternative to Burkholder's
direction. Through averaging, which is the main idea of the classical dyadic
shift, it is clear that $h_p \lesssim s^{\tmop{cl}}_p$. While the classical
shift has proven a useful model for the Hilbert transform for some
applications, it lacks a number of defining properties and similarities. 
In this paper we consider what we call the dyadic Hilbert transform
\[ \mathcal{S}_0 : h_{I_{\pm}} \mapsto \pm h_{I_{\mp}} . \]
While this operator does average to a 0 multiple of the Hilbert transform via the ideas in \cite{Pet2000} and
is therefore not applicable in the same way as the classical shift, it is in
other ways incomparably closer to the Hilbert transform. Its square is the negative identity, it is antisymmetric 
and -- apparently importantly for this subject -- it has no even component.

\

If  $\mathcal{S}_0$ has the $L^p$ bound $s_p$, it is our aim to show in this paper that Bourgain's
direction 
holds with a linear relation for this pair of operators:
\[ s_p \lesssim h_p . \]
This lower bound for the norm of Hilbert transform is surprising and new -- prior
dyadic shift arguments have all been based on convex hull arguments and
averaging, and as such they are a priori useful for upper estimates only. This
lower bound requires a deeper understanding and a better model. Indeed, the
reader will see when reading the proof, that this new shift operator mimics the
Cauchy--Riemann equations `in the probability space'. We add new elements to an
argument by {\tmname{Bourgain}} \cite{Bou1983a} using high oscillations. In his work, he
needed to apply the Hilbert transform twice to control the martingale transforms. His beautiful argument
does not seem as if it can be directly improved to change the resulting
quadratic relation. Our relationship of the Hilbert transform to this new Haar shift rather than the
martingale transform is shown to be much more direct and we therefore manage to only use
the Hilbert transform once.

\

It is also highly interesting that even operators pose fewer problems. Indeed, {\tmname{Geiss}}--{\tmname{Montgomery-Smith}}--{\tmname{Saksman}} \cite{GeiMonSak2010} proved remarkable linear two sided estimates between $m_p$ and the norm of the difference of squares of Riesz transforms in the plane, i.e. the even operator $R_1^2-R_2^2$. Part of their argument is also based on a use of high oscillations. They also gave an upper estimate for even convolution type singular integrals. We wish to highlight a strong result on linear upper estimates on even shift operators and hence more general even Calder\'on--Zygmund operators by {\tmname{Pott}}--{\tmname{Stoica}} \cite{Pot2014}.

\section{Main result}

Let us consider the dyadic system $\mathcal{D}$ and its $L^2$ normalized Haar
functions $\{h_I : I \in \mathcal{D}\}$ on the unit interval $I_0$. Analysts
write the orthonormal basis expansion of a function $f$ as follows:
\[ f (x) = \langle f \rangle_{I_0} + \sum_{I \in \mathcal{D}} (f, h_I) h_I (x)
   = \langle f \rangle_{I_0} + \frac{1}{2} \sum_{I \in \mathcal{D}} (\langle f
   \rangle_{I_+} - \langle f \rangle_{I_-})  (\chi_{I_+} - \chi_{I_-}) (x) .
\]
Any dyadic interval corresponds to a path of sign tosses, with a sequence of
$k$ sign tosses yielding an interval of length $|I_k | = 2^{- k}$ so that
$(I_1, \ldots, I_k)$ corresponds to a sequence $(\varepsilon_1, \ldots,
\varepsilon_k) \in \{-, +\}^k$. Assuming $\langle f \rangle_{I_0} = 0$ we can
relate their martingale difference sequence, using the usual notation, for example from \cite{Bou1983a},
\[ \sum^K_{k = 0} \Delta^f_k (\varepsilon_1, \ldots, \varepsilon_k)
   \varepsilon_{k + 1} \]
to the Haar series by writing
\[ \Delta^f_k (\varepsilon_1, \ldots, \varepsilon_k) = \frac{1}{2} (\langle f
   \rangle_{I_{k +}} - \langle f \rangle_{I_{k -}}) . \]
If $\varepsilon_{k + 1} = \pm$ then we end up in $I_{k \pm}$. In the above
notation we have the constant $\Delta^f_0 = \frac{1}{2} (\langle f
\rangle_{I_{0 +}} - \langle f \rangle_{I_{0 -}})$. Here, $f$ has values in a
Banach space $X$ and so $(f, h_I) \in X$.

The Banach space has the UMD property if there exists a constant $C_p$ so
that for any sign tosses $\alpha_k = \pm 1$ there holds
\[ \left\| \sum_k \alpha_{k + 1} \Delta^f_k (\varepsilon_1, \ldots,
   \varepsilon_k) \varepsilon_{k + 1} \right\|_{L_X^p} \leqslant C_p \left\|
   \sum_k \Delta^f_k (\varepsilon_1, \ldots, \varepsilon_k) \varepsilon_{k +
   1} \right\|_{L_X^p} . \]
The best such constant for $L^p$ is called the $\tmop{UMD}$ constant of $X$ and denoted here $m_p$. 
For basic properties of these spaces see the books by {\tmname{Pisier}} \cite{Pis2016} or {\tmname{Hyt\"{o}nen}}--{\tmname{van Neerven}--{\tmname{Veraar}--{\tmname{Weis} \cite{HytNeeVerWei2016a}.

In the Haar basis notation this becomes with $T_{\alpha} : \langle f
\rangle_{I_0} \mapsto 0, h_I \mapsto \alpha_I h_I$ the requirement that $\sup_{\alpha}\|
T_{\alpha} \|_{L_X^p \mapsto L_X^p} = m_p .$ Recall that with \ $\|
\mathcal{H} \|_{L_X^p \rightarrow L_X^p} = h_p$ it is known that $m_p^{1 / 2}
\lesssim h_p \lesssim m^2_p$. As mentioned before, linear relations on either
side are a famous open problem.

\

Instead of $T_{\alpha}$ we consider the shift operator $\mathcal{S}_0 : \langle
f \rangle_{I_0} \mapsto 0$, $h_{I_0} \mapsto 0$ and $S_0 : h_{I_{\pm}} \mapsto
\pm h_{I_{\mp}}$ for $I \subset I_0$. It is our goal to prove

\begin{theorem}
  \label{theorem-main}Given the $X$ valued shift operator $\mathcal{S}_0 :
  I_0 \rightarrow X$ and the Hilbert transform $\mathcal{H}: \mathbb{T}
  \rightarrow X$ with $X$ a Banach space. There exists an absolute constant
  $C$ so that
  \[ \| \mathcal{S}_0 \|_{L_X^p \rightarrow L_X^p} \leqslant C \| \mathcal{H}
     \|_{L_X^p \rightarrow L_X^p} . \]
\end{theorem}

The remaining part of this text is dedicated to the proof of Theorem
\ref{theorem-main}.

\section{Proof}

\paragraph{Using sign tosses $\varepsilon$}

Write again the identity for functions highlighting contributions from left
and right halves of intervals, in other words $\mathcal{D}^-$ and
$\mathcal{D}^+$:
\[ f = \langle f \rangle_{I_0} + (f, h_{I_0}) h_{I_0} + \sum_{k = 0}^{\infty}
   \sum_{I : |I| = 2^{- k} |I_0 |} (f, h_{I_-}) h_{I_-} + (f, h_{I_+}) h_{I_+}
   . \]
Let us now encode the sign tosses in a way that respects a shift operator of
complexity one (as opposed to complexity zero for $T_{\alpha}$). In the first
step, we just choose $\varepsilon_0 = \pm$ as the first sign toss. Then,
depending on the outcome of $\varepsilon_0$, we use generators
$\varepsilon_1^-$ and $\varepsilon_1^+$. This means, if via the previous
tosses, we arrived in a $\mathcal{D}^{\pm}$ interval, then
$\varepsilon^{\pm}_1$ is the relevant toss. They are independent with
probability 1/2 each. So the above can be rewritten as
\begin{eqnarray*}
  &  & \mathd f_{- 2} + \mathd f_{- 1} \varepsilon_0 
  \\
  &  & +\sum^{\infty}_{k = 0} \mathd f^+_k  (\varepsilon_0, \varepsilon^-_1,
  \varepsilon^+_1, \ldots, \varepsilon^+_k) \varepsilon^+_{k + 1} + \mathd
  f^-_k  (\varepsilon_0 \nocomma, \varepsilon^-_1, \varepsilon^+_1, \ldots,
  \varepsilon^+_k) \varepsilon_{k + 1}^- .
\end{eqnarray*}
\[ \  \]
We have $\mathd f_{- 2} = \langle f \rangle_{I_0}$ and $\mathd f_{- 1} = (f,
h_{I_0}) |I_0 |^{- 1 / 2}$. The second and third summands above involve
$\mathd f_0^{\pm}$, where
\[ \mathd f_0^- (\varepsilon_0) = \left\{ \begin{array}{ll}
     (f, h_{I_{0 -}}) |I_{0 -} |^{- 1 / 2} & \tmop{if} \varepsilon_0 = -\\
     0 & \tmop{if} \varepsilon_0 = +
   \end{array} \right. \]
\[ \mathd f_0^+ (\varepsilon_0) = \{ \begin{array}{ll}
     0 & \tmop{if} \varepsilon_0 = -\\
     (f, h_{I_{0 +}}) |I_{0 +} |^{- 1 / 2} & \tmop{if} \varepsilon_0 = +
   \end{array} . \]
Then $\mathd f^{\pm}_1  (\varepsilon_0, \varepsilon^+_1, \varepsilon^-_1)$
depends on $\varepsilon_0$ and depending on the outcome of $\varepsilon_0$ on
either $\varepsilon^+_1$ or $\varepsilon^-_1$. Depending on the outcome of the
relevant $\varepsilon_1$ we determine if $\mathd f_1^+$ or $\mathd f_1^-$ is
active. This is certainly not the most concise notation, but may be more
clearly resembling a sum over $\mathcal{D}_+$ and $\mathcal{D}_-$.

\paragraph{Using angles $\theta$}

Let us change to a random generator to facilitate the use of the norm of the
Hilbert transform. Let \ $\varphi_{\nosymbol}^+ (\theta_{\nosymbol}) =
\tmop{signcos} \theta_{\nosymbol}$ and $\varphi_{\nosymbol}^- (\theta) =
\tmop{signsin} \theta$. Set for $\theta_j \in [- \pi, \pi]$ with $j \in
\mathbb{N}_0$, $ \varepsilon_0= \varphi^+  (\theta_0)$ as the first sign toss.
Then, $\varepsilon_j^-=\varphi^-  (\theta_j)$ and $\varepsilon_j^+=\varphi^+  (\theta_j)$ 
for $j > 0$.  In order to get a more concise notation, we write $\theta=(\theta_0,\ldots)$ and $\vec{\theta}_k=(\theta_0,\ldots,\theta_k)$. Together we receive
\begin{eqnarray*}
  F (\theta) & = & \mathd F_{- 2} + \mathd F_{- 1} \varphi^+ (\theta_0) \\
  &  &+ \sum^{\infty}_{k = 0} \mathd F^+_k  (\vec{\theta}_k)
  \varphi^+ (\theta_{k + 1}) + \mathd F^-_k  (\vec{\theta}_k)
  \varphi^- (\theta_{k + 1})
\end{eqnarray*}
with $\mathd F_k = \mathd f_k$ for $k = - 1, - 2$ and for $k \geqslant 0$
\[ \mathd F_k^{\pm}  (\vec{\theta}_k) = \mathd f_k^{\pm} 
   (\varphi^+ (\theta_0), \varphi^- (\theta_1), \varphi^+ (\theta_1), \ldots,
   \varphi^- (\theta_k), \varphi^+ (\theta_k)) . \]

\paragraph{Fourier side, modulation and action of $\mathcal{H}$}

Let us assume for the moment that all sums are finite, including all Fourier
series if we expand in the $\theta$. This will be accomplished later by a
standard limiting procedure. With $\sigma = \pm$ observe
\[ \theta \mapsto \mathd F^{\sigma}_k  (\vec{\theta}_k)
   \varphi^{\sigma} (\theta_{k + 1}) \infixor \theta \mapsto \mathd
   F^{\sigma}_k  (\vec{\theta}_k) \mathcal{H} \varphi^{\sigma}
   (\theta_{k + 1}) \]
have a Fourier series with terms of the form
\[ e^{il_0 \theta_0} \ldots e^{il_{k + 1} \theta_{k + 1}} x_{l_0, \ldots, l_{k
   + 1}}, \]
where $x_{l_0, \ldots, l_{k + 1}} \in X$ and where these terms are summed in
the $l_j \in \mathbb{Z}$. Recall that $\int \varphi^{\sigma} (\theta_{k + 1})
= 0$ and $\int \mathcal{H} \varphi^{\sigma} (\theta_{k + 1}) = 0$ and thus
$x_{l_0, \ldots, l_k, 0} = 0$, so that we may assume in what follows $l_{k +
1} \neq 0$.

Using that all sums are finite, build a sequence $N= (N_k)_{k\geqslant0}$ so that there holds for the spectra $| \gamma_0 | + \ldots + | \gamma_k |
\leqslant N_k$  (only non-zero contributions if $|l_i |
\leqslant | \gamma_i |$). Let us define a sequence $n=(n_k)_{k\geqslant 0}$ inductively by
\[ n_0 = 1, n_{k + 1} = 2 n_k N_k . \]
Write $\vec{\theta}_k+\vec{n}_k\psi=(\theta_0 + n_0 \psi, \ldots, \theta_k + n_k \psi)$.
Then the Fourier series of
\[ \psi \mapsto \mathd F^{\sigma}_k  (\vec{\theta}_k+ \vec{n}_k \psi) \varphi^{\sigma}  (\theta_{k + 1} + n_{k + 1} \psi) \]
or
\[ \psi \mapsto \mathd F^{\sigma}_k  (\vec{\theta}_k + \vec{n}_k \psi) \mathcal{H} \varphi^{\sigma}  (\theta_{k + 1} + n_{k + 1} \psi)
\]
have summands of the form
\[ e^{il_0  (\theta_0 + n_0 \psi)} \ldots e^{il_{k + 1}  (\theta_{k + 1} +
   n_{k + 1} \psi)} x_{l_0, \ldots, l_{k + 1}} \]
with $x_{l_0, \ldots, l_k, 0} = 0$. To take the Hilbert transform in the
variable $\psi$, we must understand the sign of $l_0 n_0 + \ldots + l_k n_k +
l_{k + 1} n_{k + 1}$. Notice that, provided $l_{k + 1} \neq 0$,
\[ |l_0 n_0 + \ldots + l_k n_k | \leqslant (|l_0 | + \ldots + |l_k |) n_k
   \leqslant N_k n_k < 2 N_k n_k = n_{k + 1} \leqslant |l_{k + 1} n_{k + 1} |
\]
by the assumption on the spectrum and so the sign of $l_{k + 1}$ will
determine the sign of the sum \ $l_0 n_0 + \ldots + l_k n_k + l_{k + 1} n_{k +
1}$. For this reason $\mathcal{H}$ only sees the high frequencies of
$\varphi$. Considering just the signum of $l_{k + 1}$ means we take the
Hilbert transform of $\varphi^{\sigma}$ directly:
\begin{eqnarray}\label{modulationidentity}
  &  & \mathcal{H}_{\psi}  (\mathd F^{\sigma}_k (\vec{\theta}_k + \vec{n}_k \psi) \varphi^{\sigma} (\theta_{k + 1} + n_{k + 1} \psi))\\
  \nonumber& = & \mathd F^{\sigma}_k  (\vec{\theta}_k + \vec{n}_k \psi) \mathcal{H} (\varphi^{\sigma})  (\theta_{k + 1} + n_{k + 1} \psi) .
\end{eqnarray}
For ease of notation, we write $F^{\mathcal{H}}$ for $F$ with the Hilbert
transform applied in the last variable of each increment (such as in the
equation above) and for the modulated versions, depending upon the sequence $N
= (N_k)_{k\geqslant0}$, we write
\[ \Phi_{\theta, N} (\psi) = F (\vec{\theta} + \vec{n} \psi) \infixand \Phi^{\mathcal{H}}_{\theta, N} (\psi) = F^{\mathcal{H}} 
   (\vec{\theta} + \vec{n} \psi)  \]
   and similar for $G$ and $\Gamma$.

\paragraph{Action of $\mathcal{S}_0$}

$\mathcal{S}_0$ understood in the language of sign tosses maps as follows:
\[ \mathcal{S}_0 : \mathd f^{\pm}_k  (\varepsilon_{0,} \varepsilon^+_1,
   \varepsilon^-_1, \ldots, \varepsilon^-_k) \varepsilon^{\pm}_{k + 1} \mapsto
   \pm \mathd f^{\pm}_k  (\varepsilon^+_1, \varepsilon^-_1, \ldots,
   \varepsilon^-_k) \varepsilon_{k + 1}^{\mp} . \]
and note that in the language involving angles we have $\mathcal{S}_0
\varphi^- = - \varphi^+ \tmop{and} \mathcal{S}_0 \varphi^+ = \varphi^-$, that
is
\[ \mathcal{S}_0 \varphi^{\sigma} = \sigma \varphi^{\bar{\sigma}} \]
for $\sigma = \pm$ and $\bar{\sigma}$ the opposing sign.

We will require the following crucial lemma.

\begin{lemma}
  \label{lemma-inside-H}{\tmdummy}
  There holds
  \begin{eqnarray*} 
  & &\left| \mathbb{E}^{\theta} \langle \sum_{k = 0}^{\infty} \mathd
     F^{\sigma}_k (\vec{\theta}_k)\mathcal{H} \varphi^{\sigma}
     (\theta_{k + 1}), \sum_{l = 0}^{\infty} \mathd G^{\eta}_l (\vec{\theta}_l) \varphi^{\eta} (\theta_{l + 1}) \rangle_{X, X^{\ast}}
     \right| \\
     &\leqslant & h_p \| F \|_{L_X^p} \| G \|_{L_{X^{\ast}}^q} . 
     \end{eqnarray*}
\end{lemma}

\begin{proof}
  By a standard approximation argument, we may assume the above sums are
  finite and all functions of $\theta$ have finite spectrum. We modulate as
  described above and pull out the Hilbert transform as follows:
  \begin{eqnarray*}
    | \mathbb{E}^{\theta} \langle F^{\mathcal{H}} (\theta), G (\theta)
    \rangle_{X, X^{\ast}} | & = & | \mathbb{E}^{\psi} \mathbb{E}^{\theta}
    \langle F^{\mathcal{H}} (\theta), G (\theta) \rangle_{X, X^{\ast}} |\\
    & = & | \mathbb{E}^{\psi} \mathbb{E}^{\theta} \langle \Phi_{\theta,
    N}^{\mathcal{H}} (\psi), \Gamma_{\theta, N} (\psi) \rangle_{X, X^{\ast}} |
    .
  \end{eqnarray*}
  We tend to the second equality above. It follows immediately if we observe that $|\mathbb{E}^{\theta}  
  \langle \phi^{\mathcal{H}}_{\theta, N} (\psi), \Gamma_{\theta, N} (\psi)
  \rangle_{X, X^{\ast}} |$ does not depend upon $N$ or $\psi$. Indeed, the
  terms that arise are of the form
  \[ \langle \mathd F_k^{\sigma} (\vec{\theta}_k + \vec{n}_k \psi)\mathcal{H} \varphi^{\sigma} (\theta_{k + 1} + n_{k + 1} \psi),
     \mathd G_l^{\eta} (\vec{\theta}_l + \vec{n}_l \psi)
     \varphi^{\eta} (\theta_{l + 1} + n_{l + 1} \psi) \rangle_{X, X^{\ast}} .
  \]
  A successive integration in the
  $\theta$ shows by periodicity of the involved functions an
  independence from $\psi$ and $N$. 
  
  We now use equality (\ref{modulationidentity}) and continue to estimate
  \begin{eqnarray*}
    | \mathbb{E}^{\psi} \mathbb{E}^{\theta} \langle \Phi_{\theta,
    N}^{\mathcal{H}} (\psi), \Gamma_{\theta, N} (\psi) \rangle_{X, X^{\ast}} |
    & = & | \mathbb{E}^{\psi} \mathbb{E}^{\theta} \langle \mathcal{H}_{\psi}
    \Phi_{\theta, N} (\psi), \Gamma_{\theta, N} (\psi) \rangle_{X, X^{\ast}}
    |\\
    & \leqslant & h_p \| F \|_{L_X^p} \| G \|_{L_{X^{\ast}}^q},
  \end{eqnarray*}
  where we again used independence of the expectation of the modulations of
  $\psi$.
\end{proof}

\paragraph{Comparing $\mathcal{H}$ and $\mathcal{S}_0$ in a projected form}

Let us also define the operator $\pi$ on functions $f$ defined on $\mathbb{T}$
by $\pi (f) (x) = \sum^1_{i = - 2} \langle f \rangle_{A_i}$, where the arcs
$A_i$ correspond to angles $[i \pi / 2, i \pi / 2 + \pi / 2)$.

\begin{lemma}
  \label{lemma-projection-hilbert}There exists $c_0 > 0$ such that for both
  signatures $\sigma$ there holds
  \[ \pi \mathcal{H} \varphi^{\sigma} = c_0 \mathcal{S}_0 \varphi^{\sigma} .
  \]
\end{lemma}

\begin{proof}
  $g (x) =\mathcal{H} \chi_{(- \pi / 2, \pi / 2)}$ is an odd function with
  zeros at $\pm \pi$ and 0. There is a singularity at $\pm \pi / 2$. One can
  verify (for example by inspecting the explicit expression of the singular
  integral) that $\lim_{x \rightarrow \pi / 2} \mathcal{H} \chi_{(- \pi / 2,
  \pi / 2)} (x) = + \infty$. By translation invariance and antisymmetry of
  $\mathcal{H}$, we know that $\mathcal{H} \varphi^+ (x) = g (x) - g (x +
  \pi)$. By a similar argument, we get $\mathcal{H} \varphi^- (x) = g (x - \pi
  / 2) - g (x + \pi / 2) =\mathcal{H} \varphi^+  (x - \pi / 2)$. We also know
  that $\mathcal{H} \varphi^+ (x)$ is odd and that $\mathcal{H} \varphi^- (x)$
  is even. From this, we deduce that $\mathcal{H} \varphi^+ |_{(0, \pi)
  \setminus \{\pi / 2\}}$ is positive and symmetric about the axis $x = \pi /
  2$ and that $\mathcal{H} \varphi^+ |_{(- \pi, 0) \setminus \{- \pi / 2\}}$
  is negative and symmetric about $x = - \pi / 2$. Gathering the information,
  we obtain
  \[ \pi \mathcal{H} \varphi^+ = c_0 \varphi^- = c_0 \mathcal{S}_0 \varphi^+
     \tmop{and} \pi \mathcal{H} \varphi^- = - c_0 \varphi^+ = c_0
     \mathcal{S}_0 \varphi^- \]
  for some $c_0 > 0$. 
\end{proof}

Notice the elementary fact for functions $f, g \in L^2$:
\[ (\pi f, g)_{L^2} = (\pi f, \pi g)_{L^2} = (f, \pi g)_{L^2} . \]

\paragraph{The weak form}

For $k, l \geqslant 0$ we consider terms of the form
\begin{eqnarray*}
  &  & \mathbb{E}^{\theta}  \langle \mathcal{S}_0 \mathd F^{\sigma}_k(\vec{\theta}_k)
  \varphi^{\sigma} (\theta_{k + 1}), \mathd G_l^{\eta} (\vec\theta_l) \varphi^{\eta}
  (\theta_{l + 1}) \rangle_{X, X^{\ast}}\\
  & = & \mathbb{E}^{\theta}  \langle \mathd F_k^{\sigma} (\vec\theta_k)
  (\mathcal{S}_0 \varphi^{\sigma}) (\theta_{k + 1}), \mathd G_l^{\eta}
  (\vec\theta_l) \varphi^{\eta} (\theta_{l + 1}) \rangle_{X, X^{\ast}}\\
  & = & \mathbb{E}^{\theta}  \langle \mathd F_k^{\sigma} (\vec\theta_k) (\sigma
  \varphi^{\bar{\sigma}}) (\theta_{k + 1}), \mathd G_l^{\eta} (\vec\theta_l)
  \varphi^{\eta} (\theta_{l + 1}) \rangle_{X, X^{\ast}} .
\end{eqnarray*}
Terms where $k \neq l$ are zero after the integration in $\theta_{\max (l, k)
+ 1}$. Let us assume $l = k$. If $\sigma = \eta$, then $\bar{\sigma} \neq
\eta$ and $\mathbb{E}^{\theta_k}  (\sigma \varphi^{\bar{\sigma}})  (\theta_{k
+ 1}) \varphi^{\eta}  (\theta_{k + 1}) = 0$. So $k = l$ and $\bar{\sigma} =
\eta$ are the only arising terms.

Equally, we consider the action under $\mathcal{H}$.
\begin{eqnarray*}
  &  & \mathbb{E}^{\theta}  \langle \mathd F_k^{\sigma} (\vec\theta_k)\mathcal{H}
  \varphi^{\sigma} (\theta_{k + 1}), \mathd G_l^{\eta} (\vec\theta_l) \varphi^{\eta}
  (\theta_{l + 1}) \rangle_{X, X^{\ast}}
\end{eqnarray*}
Only diagonal terms remain because integration in $\theta_{\max (l, k) + 1}$
yields again 0 if $l \neq k$. Let now $l = k$, then $\varphi^{\eta}  (\theta_{k + 1})$ is constant on
the quarters of $[- \pi, \pi]$. Thus by Lemma \ref{lemma-projection-hilbert}
\begin{eqnarray*}
  \mathbb{E}^{\theta_{k + 1}}  (\mathcal{H} \varphi^{\sigma})  (\theta_{k +
  1}) \varphi^{\eta}  (\theta_{k + 1}) & = & \mathbb{E}^{\theta_{k + 1}}  (\pi
  \mathcal{H} \varphi^{\sigma})  (\theta_{k + 1}) \varphi^{\eta}  (\theta_{k +
  1})\\
  & = & c_0 \mathbb{E}^{\theta_{k + 1}}  (\mathcal{S}_0 \varphi^{\sigma}) 
  (\theta_{k + 1}) \varphi^{\eta}  (\theta_{k + 1})
\end{eqnarray*}
and we have together, when $\langle f \rangle_{I_0} = 0$ and $(f, h_{I_0}) =
0$
\[ \mathbb{E}^{\theta}  \langle F^{\mathcal{H}} (\theta), G (\theta)
   \rangle_{X, X^{\ast}} = c_0 \mathbb{E}^{\theta}  \langle \mathcal{S}_0 F
   (\theta), G (\theta) \rangle_{X, X^{\ast}} . \]
Putting things together, we estimate assuming $\langle f \rangle_{I_0} = 0$
and $(f, h_{I_0}) = 0$ that
\begin{eqnarray*}
  | \mathbb{E}^x \langle \mathcal{S}_0 f (x), g (x) \rangle_{X, X^{\ast}} | &
  = & |\mathbb{E}^{\theta} \langle \mathcal{S}_0 F (\theta), G (\theta)
  \rangle_{X, X^{\ast}} |\\
  & = & c^{- 1}_0  |\mathbb{E}^{\theta} \langle F^{\mathcal{H}} (\theta), G
  (\theta) \rangle_{X, X^{\ast}} |\\
  & \leqslant & h_p c^{- 1}_0 \|F\|_{L_{X,\theta}^p} \|G\|_{L_{X^{\ast},\theta}^q} \\
  &=& h_p c^{- 1}_0 \|f\|_{L_{X}^p} \|g\|_{L_{X^{\ast}}^q}.
\end{eqnarray*}
The first and last equalities hold because by construction $f(x)$ and $F(\theta)$ as well as $g(x)$ and $G(\theta)$ have the 
same probability distributions respectively. Notice that the random generators $\varphi^{\pm}$ are
independent. We have used Lemma \ref{lemma-inside-H} for the last inequality.
To finish the estimate for general $f$, write $\tilde{f} = f - (f, h_{I_0})
h_{I_0} - \langle f \rangle_{I_0} \chi_{I_0}$ and get
\[ \| \mathcal{S}_0 f \|_{L_X^p} = \| \mathcal{S}_0 \tilde{f} \|_{L_X^p}
   \leqslant c^{- 1}_0 h_p \| \tilde{f} \|_{L_X^p} \lesssim c^{- 1}_0 h_p \| f
   \|_{L_X^p} . \]
We have used that averaging operators are bounded.

Therefore, the bound of $\mathcal{S}_0$ is proportional to $h_p$ and our proof
of Theorem \ref{theorem-main} complete.

\

\end{document}